\documentclass[10pt, a4paper, reqno]{amsart}
\usepackage{amscd}
\usepackage{amsmath}
\usepackage{amstext}
\usepackage{amsthm}
\usepackage{amssymb}
\usepackage{amsopn}
\usepackage{amssymb}
\usepackage{amsxtra}
\usepackage{amsfonts}
\usepackage{fancyhdr}
\usepackage{paralist, epsfig}
\usepackage[mathcal]{euscript}
\usepackage[english]{babel}
\usepackage[latin1]{inputenc}
\usepackage{mathptmx}

\usepackage{amssymb}
\usepackage{graphicx}
\usepackage{amsmath}
\usepackage{amsfonts}

\usepackage[pdftex]{color}

\newtheoremstyle{normal}
{5pt}
{5pt}
{\normalfont}
{}
{\bfseries}
{}
{0.4em}
{\bfseries{\thmname{#1}\thmnumber{ #2}.\thmnote{ \hspace{0.5em}(#3)\newline}}}

\newtheoremstyle{kursiv}
{5pt}
{5pt}
{\itshape}
{}
{\bfseries}
{}
{0.4em}
{\bfseries{\thmname{#1}\thmnumber{ #2}.\thmnote{ \hspace{0.5em}(#3)\newline}}}

\theoremstyle{kursiv}
\newtheorem{theorem}{Theorem}
\newtheorem{lemma}[theorem]{Lemma}
\newtheorem{proposition}[theorem]{Proposition}
\newtheorem{corollary}[theorem]{Corollary}

\theoremstyle{normal}
\newtheorem{example}[theorem]{Example}
\newtheorem{remark}[theorem]{Remark}
\newtheorem{definition}[theorem]{Definition}

\newcommand{\supp}{\operatorname{supp}\nolimits}
\renewcommand{\epsilon}{\varepsilon}
\newcommand{\ind}[1]{\operatorname{ind}_{#1}\nolimits}
\newcommand{\proj}[1]{\operatorname{proj}_{#1}\nolimits}
\newcommand{\Projeins}{\ensuremath{\operatorname{Proj}^{\,\operatorname{1}}}}

\newcommand{\tensor}{\ensuremath{\otimes}}
\newcommand{\tensorepsilon}{\ensuremath{\,\otimes_{\epsilon}\,}}
\newcommand{\tensorcheck}{\ensuremath{\,\check{\otimes}_{\epsilon}\,}}
\newcommand{\tensorhat}{\ensuremath{\,\hat{\otimes}_{\pi}\,}}
\newcommand{\ep}{\ensuremath{\,\epsilon\,}}
\newcommand{\Bigcap}[1]{\ensuremath{\mathop{\cap}_{#1}}}
\newcommand{\Prod}[2]{\ensuremath{\underset{\scriptscriptstyle #1}{\overset{\scriptscriptstyle#2}{{\textstyle\prod\displaystyle}}}\,}}
\newcommand{\im}{\operatorname{im}\nolimits}

\begin{document}

\title{Weighted PLB-spaces of continuous functions arising as Tensor Products of a Fr\'{e}chet and a DF-space}

\author{Sven-Ake Wegner\,\MakeLowercase{$^{\text{a}}$}}

\renewcommand{\thefootnote}{}
\hspace{-1000pt}\footnote{\hspace{5.5pt}2010 \emph{Mathematics Subject Classification}: Primary 46E10; Secondary 46A13, 46M40, 46M05.}
\hspace{-1000pt}\footnote{\hspace{5.5pt}\emph{Key words and phrases}: PLB-space, weighted spaces of continuous functions, tensor product of a Fr\'{e}chet and a DF-space.\vspace{1.6pt}}
\hspace{-1000pt}\footnote{\hspace{5.5pt}\textit{Date:} October 10, 2010.\vspace{1.4pt}}
\hspace{-1000pt}\footnote{$^{\text{a}}$\,FB C -- Mathematics, University of Wuppertal, Gau\ss{}str.~20, 42119 Wuppertal, Germany, Phone: +49\hspace{1.2pt}202\hspace{1.2pt}439\hspace{1.2pt}253\hspace{1.2pt}1, Fax: +49\linebreak\phantom{x}\hspace{12.5pt}202\hspace{1.2pt}439\hspace{1.2pt}372\hspace{1.2pt}4, eMail: wegner@math.uni-wuppertal.de.}

\begin{abstract}
Countable projective limits of countable inductive limits, so-called PLB-spaces, of weighted Banach spaces of continuous functions have recently been investigated by Agethen, Bierstedt and Bonet, who analyzed locally convex properties in terms of the defining double sequence of weights. We complement their results by considering a defining sequence which is the product of two single sequences. By associating these two sequences with a weighted Fr\'{e}chet, resp.~LB-space of continuous functions or with two weighted Fr\'{e}chet spaces (by taking the reciprocal of one of the sequences) we derive a representation of the PLB-space as the tensor product of a Fr\'{e}chet and a DF-space and exhibit a connection between the invariants (DN) and ($\upOmega$) for Fr\'{e}chet spaces and locally convex properties of the PLB-space resp.~of the forementioned tensor product.
\end{abstract}

\maketitle

\section{Introduction}

Agethen, Bierstedt, Bonet \cite{ABB2009} investigated the structure of spaces of continuous functions defined on a locally compact and $\sigma$-compact topological space that can be written as a countable intersection of countable unions of weighted Banach spaces of continuous functions. These spaces were introduced and studied for the first time in Agethen's thesis \cite{Agethen2004}; they are examples of so-called PLB-spaces, i.e.~countable projective limits of countable inductive limits of Banach spaces. Spaces of this type arise naturally in analysis, for instance the space of distributions, the space of real analytic functions and several spaces of ultradistributions are of this type. In fact, all the forementioned spaces are even PLS-spaces, that is the linking maps in the inductive spectra of Banach spaces are compact and some of them even appear to be PLN-spaces (i.e.~the linking maps are nuclear). During the last years the theory of PLS-spaces has played an important role in the application of 
abstract functional analytic methods to several classical problems in analysis. In particular there was an intense research on tensor products of PLS-spaces motivated by the problem of parameter dependence of solutions of PDE, see e.g.~Bonet, Doma\'{n}ski \cite{BoDo2006,BoDo2008}, Doma\'{n}ski \cite{Domanski2010} and Piszczek \cite{Piszczek2010}. We refer to the survey article \cite{Domanski2004} of Doma\'{n}ski for applications, examples and further references.
\smallskip
\\Many of the applications reviewed by Doma\'{n}ski \cite{Domanski2004} are based on the theory of the so-called first derived functor of the projective limit functor. This method has its origin in the application of homological algebra to functional analysis. The research on this subject was started by Palamodov \cite{Palamodov1968,Palamodov1971} in the late sixties and carried on since the mid eighties by Vogt \cite{VogtLectures} and many others. We refer to the book of Wengenroth \cite{Wengenroth}, who laid down a systematic study of homological tools in functional analysis and in particular presents many ready-for-use results concerning concrete analytic problems. In particular, \cite[section 5]{Wengenroth} illustrates that for the splitting theory of Fr\'{e}chet or more general locally convex spaces, the consideration of PLB-spaces which are not PLS-spaces is indispensable.
\smallskip
\\Agethen, Bierstedt, Bonet \cite{ABB2009} described locally convex properties (i.e.~ultrabornologicity and barrelledness) of weighted PLB-spaces of continuous functions in terms of the defining double sequence of weights by using the theory of the derived projective limit functor $\Projeins$, which we mentioned already above. In addition they studied the interchangeability of projective and inductive limit, i.e.~the question when the PLB-spaces are equal to the weighted LF-spaces of continuous functions introduced by Bierstedt, Bonet \cite{BB1994}. Moreover, the work of Agethen, Bierstedt, Bonet exhibits that certain spaces of linear and continuous operators between K\"othe echelon spaces as well as certain tensor products of a K\"othe echelon and a coechelon space happen to be weighted PLB-spaces of continuous functions, see \cite[section 4]{ABB2009}.
\smallskip
\\In this paper we complement the results of \cite{ABB2009} by considering the following special situation. We assume that the domain of the functions in the PLB-space is the product of two topological spaces and that the defining double sequence of weights is the product of an increasing and a decreasing (single) sequence on each of the two topological spaces. By taking the reciprocal of the decreasing sequence we can associate the double sequence with two weighted Fr\'{e}chet spaces of continuous functions. This enables us to exhibit a connection between locally convex properties of the PLB-space and the invariants (DN) and ($\upOmega$) for Fr\'{e}chet spaces. The latter were introduced by Vogt \cite{Vogt1977} and Vogt, Wagner \cite{VogtWagner1980} and play a prominent role in the structure theory of Fr\'{e}chet spaces, see for instance the book \cite{MeiseVogtEnglisch} of Meise, Vogt. In particular we establish a criterion for ultrabornologicity of the PLB-spaces formulated in terms of (DN) and ($\upOmega$). Taking no reciprocals the two sequences give rise to a weighted Fr\'{e}chet and a weighted LB-space of continuous functions. In analogy to the case of sequence spaces (cf.~\cite[section 4]{ABB2009}) we show that the $\epsilon$-tensor product of these two spaces is isomorphic to the PLB-space defined by the double sequence explained above if we assume the decreasing sequence to be regularly decreasing (see Bierstedt, Meise, Summers \cite[Definition 2.1]{BMS1982}). Combining these results we finally obtain a criterion for the ultrabornologicity of the tensor product of a weighted Fr\'{e}chet space of continuous functions and a weighted LB-space of continuous functions.
\smallskip
\\The underlying general question of determining topological properties of the tensor product of a Fr\'{e}chet space and a DF-space was raised by Grothendieck in the last section of his th\`{e}se \cite{Grothendieck1955}. He investigated the case of $\pi$-tensor products of echelon and coechelon spaces of order one, see \cite[Chapitre II, \textsection 4, No. 3, Theorem 15]{Grothendieck1955}. His results inspired many authors to further studies; see Varol \cite[Section 0]{Varol2007} for references and a generalization of Grothendieck's original result to the case of $\pi$-tensor products of a K\"othe coechelon space of order one and arbitrary Fr\'{e}chet spaces (\cite[Theorem 2.1]{Varol2007}). Varol \cite[Theorem 2.7]{Varol2007} investigated in addition the case of $\epsilon$-tensor products of K\"othe coechelon spaces of order zero and arbitrary Fr\'{e}chet spaces. Note that also the classical result \cite[Chapitre II, \textsection 4, No. 3, Corollaire 2]{Grothendieck1955} (see also Bonet, P\'{e}rez Carreras \cite[Proposition 11.6.13]{BPC}) of Grothendieck on the ultrabornologicity of $s'\tensorhat{}s$ is a result on an $\epsilon$-tensor product of a Fr\'{e}chet space and a DF-space due to the nuclearity of $s$. Indeed, a certain nuclearity assumption is also the key point in results of Piszczek \cite{Piszczek2010} (cf.~also Doma\'{n}ski \cite{Domanski2010}) on tensor products of PLS-spaces. As we explain at the end of Section 4, the latter results can be used in our situation; in fact \cite[Theorem 6 and Theorem 9]{Piszczek2010} provide criteria for tensor products of FS- and LS-spaces to be ultrabornological if at least one of them is nuclear.
\smallskip
\\We refer the reader to \cite{BMS1982} for weighted spaces of continuous functions and to \cite{Jarchow,KoetheI,KoetheII,MeiseVogtEnglisch,BPC} for the general theory of locally convex spaces.

\section{Notation and Preliminary Results}\label{Notation}

Let $X$ be a locally compact and $\sigma$-compact topological space. By $C(X)$ we denote the space of all continuous functions on $X$ and by $C_c(X)$ the subspace of all functions with compact support. A weight on $X$ is a strictly positive and continuous function on $X$. We consider a double sequence $\mathcal{A}=((a_{N,n})_{n\in\mathbb{N}})_{N\in\mathbb{N}}$ of weights on $X$ which is decreasing in $n$ and increasing in $N$, i.e.~$a_{N,n+1}\leqslant a_{N,n}\leqslant a_{N+1,n}$ holds for all $N$ and $n$. This condition will be assumed on the double sequence $\mathcal{A}$ in the rest of this article. We define
$$
C(a_{N,n})_0(X):=\{\,f\in C(X)\: ; \: a_{N,n}|f| \text{ vanishes at } \infty \text{ on } X\,\}.
$$
Recall that a function $g\colon X\rightarrow \mathbb{R}$ is said to vanish at infinity on $X$ if for each $\epsilon>0$ there is a compact set $K$ in $X$ such that $|g(x)|<\epsilon$ for all $x\in X\backslash K$. The spaces $C(a_{N,n})_0(X)$ are Banach spaces for the norms $\|f\|_{N,n}:=\sup_{x\in X}a_{N,n}(x)|f(x)|$, $f\in C(a_{N,n})_0(X)$. By the definition of these norms, $C(a_{N,n})_0(X)\subseteq C(a_{N,n+1})_0(X)$ holds with continuous inclusion for all $N$ and $n$ and we can define for each $N$ the weighted inductive limit
$$
(\mathcal{A}_N)_0C(X):=\ind{n}C(a_{N,n})_0(X).
$$
The latter space needs not to be regular. By Bierstedt, Meise, Summers \cite[Theorem 2.6]{BMS1982} it is regular if and only if it is complete and this is in turn equivalent to the fact that the sequence $\mathcal{A}_N:=(a_{N,n})_{n\in\mathbb{N}}$ is regularly decreasing, see \cite[Definition 2.1]{BMS1982}. We refer to the latter article for more details on this condition. For each $N$ we have $(\mathcal{A}_{N+1})_0C(X)\subseteq(\mathcal{A}_N)_0C(X)$ with continuous inclusion. Hence, $\mathcal{A}_0C:=((\mathcal{A}_N)_0C(X))_{N\in\mathbb{N}}$ is a projective spectrum of LB-spaces with inclusions as linking maps and we can form the following projective limit, called weighted PLB-space of continuous functions
$$
(AC)_0(X):=\proj{N}(\mathcal{A}_N)_0C(X),
$$
which is the object of our study in this article. Since the space of continuous functions with compact support is contained in $(\mathcal{A}_N)_0C(X)$ and since it is dense in each Banach space $C(a_{N,n})_0(X)$, it follows that the projective limit $(AC)_0(X)$, or rather the projective spectrum $\mathcal{A}_0C$, is strongly reduced in the sense of \cite[Definition 3.3.5]{Wengenroth}. In fact it is a reduced projective limit in the sense of K\"othe \cite[p.~120]{KoetheI}, which is a stronger condition.
\medskip
\\We refer the reader to the book of Wengenroth \cite{Wengenroth} for a detailed exposition of the theory of projective spectra of locally convex spaces $\mathcal{X}=(X_N)_{N\in\mathbb{N}}$, their projective limits $\proj{N}X_N$, the derived functor $\Projeins\mathcal{X}$ and for conditions to ensure that the derived functor of a projective spectrum vanishes, i.e.~that we have $\Projeins\mathcal{X}=0$, including important results by Palamodov \cite{Palamodov1968,Palamodov1971}, Retakh \cite{Retakh1970}, Braun, Vogt \cite{BraunVogt1997}, Vogt \cite{VogtLectures,Vogt1989}, Frerick, Wengenroth \cite{FrerickWengenroth1996} and many others. At this point we only mention that, if $\mathcal{X}=(X_N)_{N\in\mathbb{N}}$ is a projective spectrum of locally convex spaces with inclusions as linking maps and limit $X=\proj{N}X_N$, the so-called fundamental resolution
$$
0\rightarrow X\rightarrow \Prod{N=1}{\infty}X_N\mathop{\rightarrow}^{\sigma}\Prod{N=1}{\infty}X_N,
$$
where $\sigma((x_N)_{N\in\mathbb{N}}):=(x_{N+1}-x_N)$, is exact but $\sigma$ is not necessarily surjective, what directs to the definition
$$
\Projeins\mathcal{X}:=\big(\Prod{N=1}{\infty}X_N\big)\big/\im\sigma.
$$
For more details see Wengenroth \cite[Chapter 3]{Wengenroth}.
\medskip
\\In \cite{Vogt1992} Vogt introduced the condition (wQ). In the case of weighted PLB-spaces one can reformulate this condition in terms of the weights as follows. We say that the sequence $\mathcal{A}$ satisfies (wQ) if
$$
\textstyle\forall \: N \; \exists \: M \geqslant N,\, n \; \forall \: K \geqslant M,\, m \; \exists \: k, \, S>0 : \frac{1}{a_{M,m}} \leqslant \max\big(\frac{S}{a_{N,n}},\frac{S}{a_{K,k}}\big).
$$
We have the following result concerning homological properties of the spectrum $\mathcal{A}_0C$ and locally convex properties of $(AC)_0(X)$, which is due to Agethen, Bierstedt, Bonet \cite{ABB2009}.

\begin{theorem}\label{thm1} (\cite[Theorem 3.7]{ABB2009}) The following are equivalent.
\\\begin{tabular}{llll}
(i) & $\Projeins \mathcal{A}_0C=0$. & (iii) & $(AC)_0(X)$ is barrelled.\\
(ii)& $(AC)_0(X)$ is ultrabornological. & (iv) & $\mathcal{A}$ satisfies condition (wQ).\hspace{104.7pt}\phantom{a}\qed
\end{tabular}
\end{theorem}

\noindent{}Given a sequence of weights $\mathcal{A}=((a_{N,n})_{n\in\mathbb{N}})_{N\in\mathbb{N}}$ we can also associate a weighted LF-space of continuous functions defined by
$$
\mathcal{V}_0C(X):=\ind{n}\proj{N}C(a_{N,n})_0(X).
$$
Spaces of this type were investigated by Bierstedt, Bonet \cite{BB1994}. It is clear, that $\mathcal{V}_0C(X)\subseteq (AC)_0(X)$ holds with continuous inclusion. The equality was investigated by Agethen, Bierstedt, Bonet \cite{ABB2009} using the following condition introduced by Vogt \cite[Satz 1.1]{Vogt1983}. We say that the sequence $\mathcal{A}$ satisfies condition (B) if
$$
\forall \: (n(N))_{N\in\mathbb{N}}\subseteq \mathbb{N} \; \exists \: m \; \forall \: M \; \exists \: L,\, c>0 \colon a_{M,m} \, \leqslant \, c \max_{\scriptscriptstyle N=1,\dots,L} a_{N,n(N)}.
$$
Agethen, Bierstedt, Bonet \cite{ABB2009} proved the following result on the forementioned equality of PLB- and LF-space.

\begin{theorem}\label{thmC}(\cite[Theorem 3.10 and Corollary 3.11]{ABB2009}) The following statements hold.
\begin{itemize}
\item[(1)] If $\mathcal{A}$ satisfies condition (B), then $(AC)_0(X)=\mathcal{V}_0C(X)$ holds algebraically. If, in addition, each $(\mathcal{A}_N)_0C(X)$ is complete, then the converse also holds.\vspace{1pt}
\item[(2)] Assume that each $(\mathcal{A}_N)_0C(X)$ is complete. Then $(AC)_0(X)=\mathcal{V}_0C(X)$ holds algebraical\-ly and topologically if and only if $\mathcal{A}$ satisfies conditions (B) and (wQ).
\hfill\qed\end{itemize}
\end{theorem}

\noindent{}In the sequel we complement the above by considering the special situation that the domain $X$ is the product of two topological spaces $X_1$ and $X_2$ and the sequence $\mathcal{A}$ is the product of an increasing sequence $(a^1_N)_{N\in\mathbb{N}}$ defined on $X_1$ and a decreasing sequence $((a^2_n)^{-1})_{n\in\mathbb{N}}$ defined on $X_2$. In this special setting we can associate a weighted Fr\'{e}chet space to each of the sequences $(a^1_N)_{N\in\mathbb{N}}$ and $(a^2_n)_{n\in\mathbb{N}}$ and thus draw the line between properties of the PLB-space and the invariants (DN) and ($\upOmega$) for Fr\'{e}chet spaces. In order to do this we have to characterize (DN) and ($\upOmega$) in terms of the weights.

\section{(DN) and ($\upOmega$) vs.~(wQ)}\label{Supplementary results: (DN) and (Omega) vs. (wQ)}

Vogt \cite{Vogt1977} and Vogt, Wagner \cite{VogtWagner1980} introduced the following conditions, which are topological invariants of Fr\'{e}chet spaces.  We say that a Fr\'{e}chet space $E$ with a fundamental sequence of seminorms $(\|\cdot\|_{n})_{n\in\mathbb{N}}$ satisfies condition (DN), if
$$
\textstyle\exists\: n\;\forall\:m\geqslant n,\,0<\theta<1\;\exists\: k\geqslant m,\,C>0\;\forall\:x\in E\,\colon \|x\|_m\leqslant C\|x\|_k^{\theta} \|x\|_n^{1-\theta}.
$$
By Meise, Vogt \cite[Lemma 29.10]{MeiseVogtEnglisch} the latter statement is equivalent to the original formulation \cite[Definition on p.~359]{MeiseVogtEnglisch} of (DN). According to \cite[Definition on p.~367]{MeiseVogtEnglisch} we say that a Fr\'{e}chet space $E$ with a fundamental sequence of seminorms $(\|\cdot\|_{n})_{n\in\mathbb{N}}$  satisfies condition ($\upOmega$), if
$$
\forall\: N\;\exists\: M\geqslant N\;\forall\: K\geqslant M\;\exists\: D>0,\,0<\theta<1\;\forall\: y\in E'\,\colon \|y\|_M^{\star}\leqslant D(\|y\|_K^{\star})^{\theta}(\|y\|_N^{\star})^{1-\theta},
$$
where $\|y\|_n^{\star}:=\sup_{\|x\|_n\leqslant1}|y(x)|$. It is well-known (and an immediate consequence of \cite[Lemma 29.13]{MeiseVogtEnglisch}) that replacing the above estimate by the inclusion $U_M\subseteq r\,U_N+D\,r^{1-1/\theta}\,U_K$ (required for each $r>0$ and with $U_n:=\{x\in E\:;\:\|x\|_n\leqslant1\}$ for $n\in\mathbb{N}$) yields an equivalent formulation of ($\upOmega$).

\begin{definition}\label{the weighted conditions} Let $A=(a_n)_{n\in\mathbb{N}}$ be an increasing sequence of weights on a Hausdorff locally compact and $\sigma$-compact topological space $X$. We say that $A$ satisfies condition $\text{(DN)}_{\text{w}}$ if
$$
\textstyle\exists\: n\;\forall\:m\geqslant n,\,0<\theta<1\;\exists\: k\geqslant m,\,C>0\,\colon a_m\leqslant C a_k^{\theta}a_n^{1-\theta}.
$$
We say that $A$ satisfies condition $\text{(}\upOmega\text{)}_{\text{w}}$ if
$$
\textstyle\forall\: N\;\exists\: M\geqslant N\;\forall\: K\geqslant M\;\exists\: D>0,\,0<\theta<1\,\colon \frac{1}{a_M}\leqslant D\big(\frac{1}{a_K}\big)^{\theta}\big(\frac{1}{a_N}\big)^{1-\theta}.
$$
\end{definition}

\begin{lemma}\label{Frechet characterization lemma 1} Let $A=(a_n)_{n\in\mathbb{N}}$ be an increasing sequence of weights on a Hausdorff locally compact and $\sigma$-compact topological space $X$. The Fr\'{e}chet space $CA_0(X)=\proj{n}C(a_n)_{0}(X)$ satisfies condition (DN) if and only if the sequence $A$ satisfies $\text{(DN)}_{\text{w}}$.
\end{lemma}
\begin{proof}\textquotedblleft{}$\Leftarrow$\textquotedblright{} We choose $n$ as in $\text{(DN)}_{\text{w}}$. For given $m\geqslant n$, $0<\theta<1$ we choose $k\geqslant m$ and $C>0$ as in $\text{(DN)}_{\text{w}}$. For an arbitrary $f\in CA_0(X)$ we obtain $\|f\|_m=\sup_{x\in X} a_m(x)|f(x)|\leqslant C\sup_{x\in X}a_k^{\theta}(x)a_n^{1-\theta}(x)|f(x)|^{\theta+1-\theta} \leqslant C\sup_{x\in X}(a_k(x)|f(x)|)^{\theta}\sup_{x\in X}(a_n(x)|f(x)|)^{1-\theta}=C\|f\|_k^{\theta}\|f\|_n^{1-\theta}$ and hence we have shown (DN).
\smallskip
\\\textquotedblleft{}$\Rightarrow$\textquotedblright{} Let $x_0\in X$ be fixed. Since $X$ is locally compact, there is a neighborhood filter $(K_{\beta})_{\beta\in B}$ for $x_0$ consisting of compact sets. Since $X$ is Hausdorff, $\Bigcap{\beta\in B}K_{\beta}=\{x_0\}$. For each $K_{\beta}$ we choose a function $f_{\beta}\in C_c(X)$ with $f_{\beta}(x_0)=1$, $0\leqslant f_{\beta}\leqslant 1$ and $\supp f_{\beta}\subseteq K_{\beta}$. Thus we have $\Bigcap{\beta\in B}\supp f_{\beta}=\{x_0\}$. Now we consider the net $(\sup_{x\in X}a(x)|f_{\beta}(x)|^{\gamma})_{\beta\in B}\subseteq \mathbb{R}$ for some fixed weight $a$ and $\gamma>0$ and obtain
$\sup_{x\in X}a(x)|f_{\beta}(x)|^{\gamma}=\sup_{x\in\supp f_{\beta}}a(x)|f_{\beta}(x)|^{\gamma}\leqslant \sup_{x\in\supp f_{\beta}}a(x)\rightarrow a(x_0)$ and $\sup_{x\in X}a(x)|f_{\beta}(x)|^{\gamma}\geqslant a(x_0)|f_{\beta}(x_0)|^{\gamma}=a(x_0)$.
\smallskip
\\Now we select $n$ as in (DN). For given $m\geqslant n$ and $0<\theta<1$ we choose $k\geqslant m$ and $C>0$ as in (DN). Now we fix some $x_0\in X$ and consider the $f_{\beta}$ defined above. We have $\sup_{x\in X}a_m(x)|f_{\beta}(x)|\leqslant C\,\sup_{x\in X}a_k^{\theta}(x)|f_{\beta}(x)|^{\theta}$\linebreak{}$\sup_{x\in X}a^{1-\theta}_n(x)|f_{\beta}(x)|^{1-\theta}$ for each $\beta\in B$. Taking limits on each side yields the desired inequality, which finishes the proof since $x_0$ was arbitrary.
\end{proof}

\begin{lemma}\label{Frechet characterization lemma 2} Let $A=(a_n)_{n\in\mathbb{N}}$ be an increasing sequence of weights on a Hausdorff locally compact and $\sigma$-compact topological $X$. The Fr\'{e}chet space $CA_0(X)=\proj{n}C(a_n)_{0}(X)$ satisfies condition ($\Omega$) if and only if the sequence $A$ satisfies $\text{(}\Omega\text{)}_{\text{w}}$.
\end{lemma}
\begin{proof}\textquotedblleft{}$\Leftarrow$\textquotedblright{} We put $a:=\frac{1}{a_N}$ and $b:=\frac{1}{a_K}$ and use (cf.~\cite[proof of Lemma 29.13]{MeiseVogtEnglisch}) that $\min_{s>0}(sa+s^{1-1/\theta}\,b)=\theta^{-1}(\theta^{-1}-1)^{\theta-1}a^{1-\theta}\,b^{\theta}$ holds for arbitrary $a$ and $b>0$. We put $C:=\theta^{-1}(\theta^{-1}-1)^{\theta-1}>0$ and thus get $a^{1-\theta}\,b^{\theta}\leqslant \frac{1}{C}\,(sa+s^{1-1/\theta}\,b)$ for each $s>0$. Hence we obtain
$D\,\big(\frac{1}{a_K}\big)^{\theta}\big(\frac{1}{a_N}\big)^{1-\theta}\leqslant \frac{D}{C}\big(s\,\frac{1}{a_N}+s^{1-1/\theta}\,\frac{1}{a_K}\big)\leqslant \max\big(\,2\,\frac{D}{C}\,s\,\frac{1}{a_N},\,2\,\frac{D}{C}\,s^{1-1/\theta}\,\frac{1}{a_K}\,\big)$ for each $s>0$. Now we define $r=4\,\frac{D}{C}\,s$ and $D'=\frac{4D}{C}(\frac{C}{4D})^{1-1/\theta}$\linebreak{}$>0$, hence $s=\frac{C}{4D}r$ and therefore $2\frac{D}{C}s^{1-1/\theta}=\frac{4D}{2C}(\frac{C}{4D})^{1-1/\theta}\,r^{1-1/\theta}=\frac{D'}{2}\,r^{1-1/\theta}$. Replacing $D'$ with $D$ we get
$$
(\star)\;\;\;\;\;\;\textstyle\forall\: N\;\exists\: M\geqslant N\;\forall\: K\geqslant M\;\exists\: D>0,\,0<\theta<1\;\forall\:r>0\,\colon \frac{1}{a_M}\leqslant \max\big(\frac{r}{2a_N},\,\,\frac{D\,r^{1-1/\theta}}{2a_K}\big).
$$
Now we show ($\upOmega$) in the second formulation mentioned at the beginning of this section. Let $N$ be given. We choose $M\geqslant N$ as in $(\star)$. For given $K\geqslant M$ we select $D>0$ and $0<\theta<1$ as in $(\star)$ and take an arbitrary $r>0$. Let $f\in U_M$ be fixed, i.e.~$|f|\leqslant \frac{1}{a_M}\leqslant \max\big(\frac{r}{2a_N},\,\,\frac{D\,r^{1-1/\theta}}{2a_K}\big)$. By \cite[Lemma 3.4]{ABB2009} there exist $\varphi_1,\,\varphi_2\in C(X)$ with $0\leqslant\varphi_1,\,\varphi_2\leqslant1$, $\varphi_1+\varphi_2=f$ such that $|\varphi_1|\leqslant\frac{r}{a_N}$, $|\varphi_2|\leqslant\frac{Dr^{1-1/\theta}}{a_K}$, i.e.~$\varphi_1\in rU_N$, $\varphi_2\in Dr^{1-1/\theta}U_K$ and thus $f\in rU_N+Dr^{1-1/\theta}U_K$.
\smallskip
\\\textquotedblleft{}$\Rightarrow$\textquotedblright{} For a fixed $N$ and for $x_0\in X$ we consider $\delta_{x_0}\colon C(a_n)_{0}(X)\rightarrow \mathbb{C}$, $\delta_{x_0}(f):=f(x_0)$. Then we have $\|\delta_{x_0}\|_N^{\star}=\sup_{f\in U_N}|\delta_{x_0}(f)|=\sup_{f\in U_N}|f(x_0)|\leqslant\textstyle \frac{1}{a_N(x_0)}$. We choose $\varphi\in C_c(X)$ with $\varphi(x_0)=1$, $0\leqslant \varphi \leqslant 1$ on $X$ and put $f_0:=\frac{\varphi}{a_N}$. Then $f_0\in C(a_N)_{0}(X)$ and $\sup_{x\in X}a_N(x)|f_0(x)|$ $=\sup_{x\in X}a_N(x){\textstyle\frac{\varphi(x)}{a_N(x)}}=\sup_{x\in X}\varphi(x)=1$, i.e.~$f_0\in U_N$. $\delta_{x_0}(f_0)=f_0(x_0)=\frac{\varphi(x_0)}{a_N(x_0)}=\frac{1}{a_N(x_0)}$ implies $\|\delta_{x_0}\|_N^{\star}=\frac{1}{a_N(x_0)}$. $\text{(}\upOmega\text{)}_{\text{w}}$ is the special case of ($\upOmega$) where we choose the functional to be $\delta_{x}$ for arbitrary $x\in X$.
\end{proof}

\noindent{}After these preparations we are ready to investigate the situation we mentioned at the end of Section 2: For $i\in\{1,2\}$ let $X_i$ denote a locally compact and $\sigma$-compact Hausdorff topological space. Moreover, let $A^i=(a^i_n)_{n\in\mathbb{N}}$ be an increasing sequence of weights on $X_i$, i.e.~$a^i_n\leqslant a^i_{n+1}$ for all $n\in\mathbb{N}$. We define the double sequence $\mathcal{A}=((a_{N,n})_{N\in\mathbb{N}})_{n\in\mathbb{N}}$ by setting
$$
\textstyle a_{N,n}\,\colon X_1\times X_2\rightarrow\mathbb{R},\;\; (x_1,x_2)\mapsto a_{N,n}(x_1,x_2):=\big[a^1_N\tensor \frac{1}{a^2_n}\big](x_1,x_2)=\frac{a^1_N(x_1)}{a^2_n(x_2)}.
$$
Thus, $\mathcal{A}$ satisfies the estimates $a_{N,n+1}(x_1,x_2)\leqslant a_{N,n}(x_1,x_2)\leqslant a_{N+1,n}(x_1,x_2)$ for all $N$, $n$ and $(x_1,x_2)\in X_1\times X_2$. In the sequel we refer to a sequence of the latter form by $\mathcal{A}=A^1\tensor{}(A^2)^{-1}$. To simplify notation we put $X:=X_1\times X_2$ and consider the PLB-space $(AC)_0(X)$ in the notation established at the beginning of Section 2. In view of Theorem \ref{thm1} we investigate if there is some relation between the conditions (DN) and ($\upOmega$) for the Fr\'{e}chet spaces $C(A^i)_0(X_i)$ and the PLB-space $(AC)_0(X)$. According to the results above we can consider the weight conditions $\text{(DN)}_{\text{w}}$, $\text{(}\upOmega\text{)}_{\text{w}}$ and (wQ).

\begin{proposition}\label{Fundamental Lemma} Let $A^1$ resp.~$A^2$ be an increasing sequence of weights on $X_1$ resp.~$X_2$. Assume that $A^1$ satisfies $\text{(}\Omega\text{)}_{\text{w}}$ and that $A^2$ satisfies $\text{(DN)}_{\text{w}}$. Then the sequence $\mathcal{A}=A^1\tensor (A^2)^{-1}$ on $X_1\times{}X_2$ satisfies (wQ).
\end{proposition}
\begin{proof}It suffices to show
$$
\textstyle\forall\: N\;\exists\:M\geqslant N,\,n\;\forall\:K\geqslant M,\,m\;\exists\:k,\,S>0\;\forall\:(x_1,x_2)\in X_1\times X_2\,\colon\textstyle\frac{a_m^2(x_2)}{a_M^1(x_1)}\leqslant S \max\big( \frac{a^2_n(x_2)}{a^1_N(x_1)},\frac{a^2_k(x_2)}{a^1_K(x_1)} \big).
$$
In order to do this, let $N$ be given. We select $M\geqslant N$ as in $\text{(}\upOmega\text{)}_{\text{w}}$ and $n$ as in $\text{(DN)}_{\text{w}}$. For given $K\geqslant M$ there exist $D>0$ and $0<\theta<1$ with the estimate in $\text{(}\upOmega\text{)}_{\text{w}}$. For arbitrary $m$ and the same $\theta$ there exist $k\geqslant m$ and $C>0$ with the estimate in $\text{(DN)}_{\text{w}}$. We put $S:=2CD$ and multiply the estimates in $\text{(DN)}_{\text{w}}$ and $\text{(}\upOmega\text{)}_{\text{w}}$ to get $\frac{a^2_m(x_2)}{a^1_M(x_1)}\leqslant CD (\frac{a^2_k(x_2)}{a^1_K(x_1)})^{\theta}(\frac{a^2_n(x_2)}{a^1_N(x_1)})^{1-\theta}\leqslant CD(\frac{a^2_k(x_2)}{a^1_K(x_1)}+\frac{a^2_n(x_2)}{a^1_N(x_1)})\leqslant S\max(\frac{a^2_k(x_2)}{a^1_K(x_1)},\,\frac{a^2_n(x_2)}{a^1_N(x_1)})$ for each $(x_1,x_2)\in X_1\times X_2$.
\end{proof}

\noindent{}In the light of Theorem \ref{thmC} it is natural to ask wether there is any relation between the conditions (wQ) and (B); this was one of the objectives in \cite{DA}. In the latter work a variant of the following example was stated to show that \textquotedblleft{}(wQ)$\Rightarrow$(B)\textquotedblright{} is not true in general. Proposition \ref{Fundamental Lemma} yields a very easy way of concluding that the double sequence $\mathcal{A}$ in Example \ref{counter example for (B)} satisfies (wQ); we believe that this is less complicated than the approach of \cite{DA}.

\begin{example}\label{counter example for (B)}In general, the assumptions of Proposition \ref{Fundamental Lemma} do not imply that $\mathcal{A}$ satisfies condition (B). Let $s$ be the Fr\'{e}chet space of rapidly decreasing sequences. Then $s= \{x\in \mathbb{K}^{\mathbb{N}} \; ; \; \forall \: k\in \mathbb{N}\,\colon\lim_{j\rightarrow\infty} j^k|x_j|=0\}=CA_0(\mathbb{N})$ holds  for $A=(a_k)_{k\in\mathbb{N}}$ with $a_k(j)=j^k$ by \cite[Proposition 28.16]{MeiseVogtEnglisch}. In particular, $s$ satisfies (DN) and ($\upOmega$), see e.g.~Meise, Vogt \cite[Section III.29]{MeiseVogtEnglisch}. Now we consider $\mathcal{A}=((a_N\otimes{}a_n^{-1})_{n\in\mathbb{N}})_{N\in\mathbb{N}}$ on $\mathbb{N}\times\mathbb{N}$, i.e.~$(AC)_0(\mathbb{N}\times\mathbb{N})=\proj{N}\ind{n}C(a_N\otimes{}a_n^{-1})_0(\mathbb{N}\times{}\mathbb{N})$, which satisfies (wQ) by Proposition \ref{Fundamental Lemma}. Straight forward calculations show that $\mathcal{A}$ does not satisfy (B).
\end{example}

\begin{theorem}\label{Main result} Let $A^1$ resp.~$A^2$ be an increasing sequence of weights on $X_1$ resp.~$X_2$. Consider the sequence $\mathcal{A}=A^1\tensor(A^2)^{-1}$ on $X=X_1\times{}X_2$ and assume that $C(A^1)_0(X_1)$ satisfies ($\Omega$) and that $C(A^2)_0(X_2)$ satisfies (DN). Then $(AC)_0(X)$ is ultrabornological.
\end{theorem}
\begin{proof} The result follows immediately from Lemmas \ref{Frechet characterization lemma 1} and \ref{Frechet characterization lemma 2}, Proposition \ref{Fundamental Lemma} and Theorem \ref{thm1}.
\end{proof}

\begin{remark}\label{DN-bar-Omega-bar-bar-remark} Changing certain quantifiers in the definitions of (DN) and ($\upOmega$), the latter turn into the weaker resp.~stronger conditions (\underline{DN}) and ($\overline{\overline{\upOmega}}$). For the definition and recent results involving (\underline{DN}) and ($\overline{\overline{\upOmega}}$) we refer to Bonet, Doma\'{n}ski \cite{BoDo2001,BoDo2006,BoDo2008}.
\smallskip
\\Having (\underline{DN}) and ($\overline{\overline{\upOmega}}$) at hand it is easy to see how weight conditions $\text{(}\underbar{\text{DN}}\text{)}_{\text{w}}$ and $\text{(}\overline{\overline{\upOmega}}\text{)}_{\text{w}}$ have to be defined in order to characterize the invariants (\underline{DN}) and ($\overline{\overline{\upOmega}}$) for weighted Fr\'{e}chet spaces $CA_0(X)$. Then it is not hard to see that Proposition \ref{Fundamental Lemma} and Theorem \ref{Main result} are also valid with (DN) resp.~($\upOmega$) replaced by (\underline{DN}) resp.~($\overline{\overline{\upOmega}}$).
\smallskip
\\Since a detailed proof for these variants of Proposition \ref{Fundamental Lemma} and Theorem \ref{Main result} is straight forward, we ommit it at this point; a detailed exposition can be found in \cite{PHD}. However, the results might be of interest in view of applications and examples.
\end{remark}

\section{Tensor product representation}\label{Tensor product representation}

The constructions in the earlier sections already suggest the question, wether the space $(AC)_0(X)$ with an underlying sequence $\mathcal{A}=A^1\tensor{}(A^2)^{-1}$ can be realized as the tensor product of a Fr\'{e}chet and an LB-space. In the last section we deduce a representation of this kind, which finally will enable us (see Theorem \ref{(F)-(DF)-ub-cor}) to utilize Theorem \ref{Main result} to prove a criterion for the ultrabornologicity of an $\epsilon$-tensor product of a weighted Fr\'{e}chet space of continuous functions and a weighted LB-space of continuous functions. For further information on the general question of determining locally convex properties of tensor products of a Fr\'{e}chet and a DF-space we refer to our comments in Section 1.
\smallskip
\\In what follows we need the so-called regularly decreasing condition; we refer to Bierstedt, Meise, Summers \cite[Definition 2.1]{BMS1982} for the definition and further information.

\begin{proposition}\label{representation theorem} Let $A^1$ resp.~$A^2$ be an increasing sequence of weights on $X_1$ resp.~$X_2$. Consider the double sequence $\mathcal{A}=A^1\tensor(A^2)^{-1}$ on $X=X_1\times{}X_2$ and assume that $\mathcal{V}^2=(A^2)^{-1}$ is regularly decreasing. Then we have the isomorphism
$$
C(A^1)_0(X_1)\tensorcheck \mathcal{V}_0^2C(X_2)\cong (AC)_0(X),
$$
where $C(A^1)_0(X_1)=\proj{n}C(a^1_N)_0(X_1)$ resp.~$\mathcal{V}_0^2C(X_2)=\ind{n}C((a^2_n)^{-1})_0(X_2)$ is a weighted Fr\'{e}chet resp. LB-space of continuous functions.
\end{proposition}
\begin{proof}We compute
\begin{align*}
C(A^1)_0(X_1)\tensorcheck \mathcal{V}_0^2C(X_2) &\stackrel{\text{\tiny{dfn}}}{=} \big(\proj{N}C(a^1_N)_0(X_1)\big)\tensorcheck\big(\ind{n}C((a^2_n)^{-1})_0(X_2)\big)\\
&\stackrel{\text{\tiny(1)}}{=} \proj{N}\big[C(a_N^1)_0(X_1) \tensorcheck \ind{n}C((a^2_n)^{-1})_0(X_2)\big]\\
&\stackrel{\text{\tiny(2)}}{=} \proj{N}\ind{n}\big[C(a_N^1)_0(X_1) \tensorcheck C((a^2_n)^{-1})_0(X_2)\big]\\
&\stackrel{\text{\tiny(3)}}{=} \proj{N}\ind{n}\big[C(a_N^1\tensor (a_n^2)^{-1})_0(X_1\times X_2)\big]\\
&\stackrel{\text{\tiny{dfn}}}{=} (AC)_0(X).
\end{align*}
The isomorphy (1) is true in general, see e.g.~Jarchow \cite[16.3.2]{Jarchow}.
\smallskip
\\$(2)$ can be seen as follows. We have
\begin{itemize}
\item[a.] \cite[Theorem 4.1]{Hollstein1980}: $C(a_N^1)_0(X_1)$ is an $\ep$-space by Hollstein \cite[Proposition 2.3]{Hollstein1980},
\item[b.] \cite[Theorem 4.1.(ii)]{Hollstein1980}: $\mathcal{V}_0^2C(X_2)$ is quasi-complete, compact-regular (cf.~Bierstedt, Mei\-se, Summers \cite[Corollary 2.7]{BMS1982}) and all the $C((a^2_n)^{-1})_0(X_2)$ have the approximation property (cf.~Bierstedt \cite[Theorem 5.5.(3)]{Bierstedt1973}),
\item[c.] \cite[Proposition 4.4.(1)]{Hollstein1980}: $C(a_N^1)_0(X_1)$ is Banach and $\mathcal{V}_0^2C(X_2)$ is compact-regular (see b.),
\item[d.] \cite[Theorem 4.1., 2nd part]{Hollstein1980}: $\mathcal{V}_0^2C(X_2)$ and all the $C((a^2_n)^{-1})_0(X_2)$ are complete (see b.),
\end{itemize}
Therefore by Hollstein \cite[Theorem 4.1]{Hollstein1980} we have an isomorphism
$$
\ind{n}\big[C(a_N^1)_0(X_1) \tensorcheck C((a^2_n)^{-1})_0(X_2)\big]\stackrel{\sim}{\longrightarrow} C(a_N^1)_0(X_1) \tensorcheck \ind{n}C((a^2_n)^{-1})_0(X_2)
$$
for each $N$. It is easy to see that the latter yields an isomorphism of the corresponding projective limits and we therefore get the desired isomorphism (2).
\smallskip
\\Finally, a result of Bierstedt \cite[Theorem 1.2]{Bierstedt1974} yields an isomorphism
$$
C(a_N^1)_0(X_1) \tensorcheck C((a^2_n)^{-1})_0(X_2)\stackrel{\sim}{\longrightarrow}C(a_N^1\tensor (a_n^2)^{-1})_0(X_1\times X_2),
$$
for each $N$ and $n$ which is given by $\textstyle\sum_{i=1}^jf_j\tensor g_i\mapsto\big[(x_1,x_2)\mapsto \sum_{i=1}^jf_i(x_1)g_i(x_2)\big]$ -- note that the equality $C(a_N^1)_0(X_1) \tensorcheck C((a^2_n)^{-1})_0(X_2)=C(a_N^1)_0(X_1)$ $\tensorepsilon C((a^2_n)^{-1})_0(X_2)$ follows from K\"othe \cite[\textsection{}44, 2.(5)]{KoetheII}, since both spaces are complete and $C((a^2_n)^{-1})_0(X_2)$ has the approximation property (see a.~above). Again it is easy to see that the above isomorphisms yield an isomorphism of the corresponding PLB-spaces. This finally shows $(3)$.
\end{proof}

\noindent{}In the case of sequence spaces, the above tensor product is of the type $\lambda^0(A)\tensorcheck k^0(B)$, i.e.~it is the tensor product of a K\"othe echelon and K\"othe coechelon space and thus Proposition \ref{representation theorem} can be regarded as an extension of \cite[Lemma 4.3]{ABB2009} to continuous functions.
\smallskip
\\The space considered in Example \ref{counter example for (B)}, $(AC)_0(\mathbb{N}\times\mathbb{N})$ with $\mathcal{A}=((a_N\tensor{}a_n^{-1})_{n\in\mathbb{N}})_{N\in\mathbb{N}}$ and $a_k(j)=j^k$, is by Proposition \ref{representation theorem} isomorphic to the space $s\tensorcheck k^0(B)$ where $B=(j^{-k})_{j,k\in\mathbb{N}}$ (in this case the regularly decreasing condition \cite[Definition 3.1]{BMS1982a} can easily be verified). In fact, $(AC)_0(\mathbb{N}\times\mathbb{N})$ is isomorphic to $s\tensorcheck{}s'_b$ as Corollary \ref{COR-2} below will show.

\begin{theorem}\label{(F)-(DF)-ub-cor}Let $A^1$ resp.~$A^2$ be an increasing sequence of weights on $X_1$ resp.~$X_2$. Assume that $\mathcal{V}^2=(A^2)^{-1}$ is regularly decreasing, that $C(A^1)_0(X_1)$ satisfies ($\Omega$) and that $C(A^2)_0(X_2)$ satisfies (DN). Then the $\ep$-tensor product $C(A^1)_0(X_1)\tensorcheck \mathcal{V}_0^2C(X_2)$ of a Fr\'{e}chet space and a DF-space is ultrabornological.
\end{theorem}
\begin{proof} The result follows directly from Theorem \ref{Main result} and Proposition \ref{representation theorem}.
\end{proof}

\begin{remark}\label{(F)-(DF)-ub-cor-rem}In view of Remark \ref{DN-bar-Omega-bar-bar-remark} it is clear that Theorem \ref{(F)-(DF)-ub-cor} is also valid if we replace (DN) resp.~($\upOmega$) by (\underline{DN}) resp.~($\overline{\overline{\upOmega}}$). Again we refer to \cite{PHD} for details.\hfill\qed
\end{remark}

\noindent{}In the case of sequence spaces the meaning of Theorem \ref{(F)-(DF)-ub-cor} and Remark \ref{(F)-(DF)-ub-cor-rem} is the following: Let $A=(a_n)_{n\in\mathbb{N}}$ be a K\"othe matrix and $B=(b_n)_{n\in\mathbb{N}}$ be a decreasing sequence of strictly positive functions on $\mathbb{N}$ which is regularly decreasing. Put $B^{-1}=(b^{-1}_n)_{n\in\mathbb{N}}$. Assume that $\lambda^0(A)$ satisfies ($\upOmega$) and $\lambda^0(B^{-1})$ satifies (DN) or that $\lambda^0(A)$ satisfies ($\overline{\overline{\upOmega}}$) and $\lambda^0(B^{-1})$ satisfies (\underline{DN}). Then the two statements above imply that the space $\lambda^0(A)\tensorcheck{}k^0(B)$ is ultrabornological. The latter follows also from the results in \cite[Section 4]{ABB2009} and Proposition \ref{Fundamental Lemma}, since in \cite[Section 4]{ABB2009} ultrabornologicity of the space $\lambda^0(A)\tensorcheck{}k^0(B)$ was characterized via condition (wQ). Let us add that in the case of sequence spaces, $B$ is regularly decreasing if and only 
if $\lambda^0(B^{-1})$ is quasinormable and that this is equivalent to condition (wS) of Bierstedt, Meise, Summers, see \cite[Proposition on p.~48 and Proposition 3.2]{BMS1982a}.
\bigskip
\\The setup of this section, in particular the fact that in Theorem \ref{(F)-(DF)-ub-cor} and Remark \ref{(F)-(DF)-ub-cor-rem} one of the assumptions is an assumption on the space $C(A^2)_0(X_2)$ whereas the conclusion affects the space $\mathcal{V}^2_0C(X_2)$ with $\mathcal{V}^2=(A^2)^{-1}$, suggests the question if in the latter statements  $\mathcal{V}^2_0C(X_2)$ can be replaced by $C(A^2)_0(X_2)'_b$ or if dually $C(A^2)_0(X_2)$ can be replaced by $\mathcal{V}^2_0C(X_2)'_b$.
\smallskip
\\For sequence spaces we have (in the notation above) $k^0(B)'_b\cong{}\lambda^1(B^{-1})$, see \cite[Proposition 2.10]{Bierstedt1988}. If now $k^0(B)'_b$ satisfies (DN), the same is true for $\lambda^1(B^{-1})$ and we can utilize \cite[Exercise 29.8]{MeiseVogtEnglisch} in combination with arguments similar to the proof of \cite[Lemma 29.10]{MeiseVogtEnglisch} to obtain that the K\"othe matrix $B^{-1}$ satisfies $\text{(DN)}_{\text{w}}$. Then Lemma \ref{Frechet characterization lemma 1} implies that $\lambda^0(B^{-1})$ satisfies (DN). Thus, in the case of sequence spaces our second question can be answered positively as follows.

\begin{corollary}\label{COR-1} Let $A=(a_n)_{n\in\mathbb{N}}$ be a K\"othe matrix and $B=(b_n)_{n\in\mathbb{N}}$ be a decreasing sequence of strictly positive functions on $\mathbb{N}$ which is regularly decreasing. Assume that $\lambda^0(A)$ satisfies ($\Omega$) and $k^0(B)'_b$ satifies (DN) or that $\lambda^0(A)$ satisfies ($\overline{\overline{\Omega}}$) and $k^0(B)'_b$ satisfies (\underline{DN}). Then the space $\lambda^0(A)\tensorcheck{}k^0(B)$ is ultrabornological.\hfill\qed
\end{corollary}

\noindent{}In view of the second question, we have $\lambda^0(B^{-1})'_b\cong{}k^1(B)$, see \cite[Proposition 2.10]{Bierstedt1988}. Moreover, if $\lambda^0(B^{-1})$ or equivalently $\lambda^0(B^{-1})'_b$, is nuclear then $k^1(B)\cong{}k^0(B)$ is valid (cf.~\cite[remark after Proposition 2.15]{Bierstedt1988}) and hence also our first question admits a positive answer in the case of (nuclear) sequence spaces. Let us add at this point that nuclearity of $\lambda^0(B^{-1})$ is equivalent to $B$ satisfying the Grothendieck-Pietsch condition (N), see \cite[Remark after Theorem 5.4]{BMS1982a}, which implies the regularly decreasing condition by \cite[Remark after Theorem 5.4 and Proposition 4.8]{BMS1982a}.

\begin{corollary}\label{COR-2} Let $A=(a_n)_{n\in\mathbb{N}}$ and $B=(b_n)_{n\in\mathbb{N}}$ be K\"othe matrices such that $\lambda^0(B)$, or equivalently its strong dual, is nuclear. Assume that $\lambda^0(A)$ satisfies ($\Omega$) and $\lambda^0(B)$ satifies (DN) or that $\lambda^0(A)$ satisfies ($\overline{\overline{\Omega}}$) and $\lambda^0(B)$ satisfies (\underline{DN}). Then the space $\lambda^0(A)\tensorcheck{}\lambda^0(B)'_b$ is ultrabornological.\hfill\qed
\end{corollary}

\noindent{}In the general case of continuous functions, an analogon of Corollary \ref{COR-1} can be proved under the additional assumption that both factors of the tensor product are Schwartz and that one of them is nuclear -- this is achieved by using the results \cite[Theorem 6 and Theorem 9]{Piszczek2010} of Piszczek (compare with Doma\'{n}ski \cite[Corollary 5.6]{Domanski2010}) on tensor products of PLS-spaces.

\begin{proposition}\label{Schwartz-Prop}Let $A$ be an increasing and $\mathcal{V}$ be a decreasing sequence of weights on $X_1$ resp.~$X_2$. Assume that the spaces $CA_0(X_1)$ and $\mathcal{V}_0C(X_2)$ are Schwartz and that one of them is nuclear. If $CA_0(X_1)$ satisfies ($\Omega$) and $\mathcal{V}_0C(X_2)'_b$ satifies (DN) or $CA_0(X_1)$ satisfies ($\overline{\overline{\Omega}}$) and $\mathcal{V}_0C(X_2)'_b$ satisfies (\underline{DN}), then the $\epsilon$-tensor product $CA_0(X_1)\tensorcheck{}\mathcal{V}_0C(X_2)$ is ultrabornological.
\end{proposition}
\begin{proof}

\noindent{}Our general assumptions imply that $CA_0(X_1)$ and $\mathcal{V}_0C(X_2)$ both enjoy the dual interpolation estimate for big $\theta$ (in the case of (DN) and (${\upOmega}$)) or small $\theta$ (in the case of (\underline{DN}) and ($\overline{\overline{{\upOmega}}}$)), cf.~\cite[Proposition 1]{Piszczek2010}. With \cite[Theorem 6 resp.~Theorem 9]{Piszczek2010} our nuclearity assumptions imply that $CA_0(X_1)\tensorcheck{}\mathcal{V}_0C(X_2)$ has the dual interpolation estimate for small or for big $\theta$. Since $CA_0(X_1)$ and $\mathcal{V}_0C(X_2)$ are ultrabornological PLS-spaces, their completed $\epsilon$-tensor product is a PLS-space, cf.~\cite[remarks at the end of Section 2]{Piszczek2010}. Thus, $CA_0(X_1)\tensorcheck{}\mathcal{V}_0C(X_2)$ is ultrabornological by Bonet, Doma\'{n}ski \cite[Corollary 1.2.(c)]{BoDo2007}.
\end{proof}

\noindent{}In view of the arguments previous to Corollary \ref{COR-1}, the sequence space case of Proposition \ref{Schwartz-Prop} implies the sequence space cases of Theorem \ref{(F)-(DF)-ub-cor} and Remark \ref{(F)-(DF)-ub-cor-rem}, if we assume that $\lambda^0(A)$ and $\lambda^0(B^{-1})$ are Schwartz and that one of them is nuclear. In fact, the assumptions of Proposition \ref{Schwartz-Prop} already imply $X_1\cong{}X_2\cong\mathbb{N}$ as the next statement shows.

\begin{remark}\label{Schwartz-discrete} Let $X$ be a Hausdorff locally compact and $\sigma$-compact space. Let $A=(a_n)_{n\in\mathbb{N}}$ be an increasing sequence of weights on $X$ and let $CA_0(X)=\proj{n}C(a_n)_0(X)$ be Schwartz. Then $X\cong\mathbb{N}$, where $\mathbb{N}$ is endowed with the discrete topology. The same conclusion is valid if we replace $CA_0(X)$ by $\mathcal{V}_0C(X)=\ind{n}C(v_n)_0(X)$ for a decreasing sequence $\mathcal{V}=(v_n)_{n\in\mathbb{N}}$ of weights.
\end{remark}
\begin{proof} Let $K\subseteq X$ be compact and consider the restriction map $q\colon CA_0(X)\rightarrow C(K)$. With the definition $C:=\sup_{x\in K}1/a_1(x)$ we have $a_1C\geqslant1$ on $K$ and thus $\sup_{x\in K}|f(x)|\leqslant C\sup_{x\in K}a_1(x)|f(x)|\leqslant{}C\sup_{x\in X}a_1(x)|f(x)|$ for each $f\in CA_0(X)$, i.e.~$q$ is continuous if we endow $C(K)$ with the sup-norm. Hence, $q$ induces an isomorphism $CA_0(X)/\ker{}q\cong{}C(K)$ and thus $C(K)$ is a Schwartz space by \cite[24.18]{MeiseVogtEnglisch}. But $C(K)$ is also a Banach space, whence it is finite dimensional which in turn implies that $K$ is finite (for infinite $K$ it is not hard to construct an infinite set of continuous functions which is linearly independent); in particular, $K$ has to carry the discrete topology. Now the conclusion follows since we assumed that $X$ is $\sigma$-compact. In the case of $\mathcal{V}_0C(X)$ we may use the same arguments; the continuity of the restriction maps follows from the universal property of the inductive limit (e.
g.~\cite[24.7]{MeiseVogtEnglisch}).
\end{proof}

\noindent{}To conclude, let us point out that Theorem \ref{(F)-(DF)-ub-cor} and Remark \ref{(F)-(DF)-ub-cor-rem} clearly are more general in the case of weighted spaces of continuous functions whereas the results of Piszczek and Doma\'{n}ski which leaded us to Proposition \ref{Schwartz-Prop} provide a general theory for (nuclear) PLS-space.\medskip

\bigskip

\footnotesize

{\sc Acknowledgements. }This article arises from a part of the author's doctoral thesis which was supervised by Klaus D.~Bierstedt and Jos\'{e} Bonet. The author thanks both of them for all their advice, helpful remarks and valuable suggestions. In addition, the authour likes to thank the referee for the correction of an error in Lemma \ref{Frechet characterization lemma 2}, for pointing out the relation of this work to the recent papers of Piszczek and Doma\'{n}ski and for his comments on the authors original proof of Proposition \ref{Fundamental Lemma} (which was inspired by that of Vogt \cite[Theorem 5.1]{Vogt1987} and involved slightly different but equivalent versions of $\text{(DN)}_{\text{w}}$ and $\text{(}\upOmega\text{)}_{\text{w}}$), which leaded to a simplification of the presented results and in particular of the forementioned proof.

\end{document}